\documentclass[a4paper,11pt]{scrartcl}
\usepackage{anyfontsize}
\usepackage{lmodern}
\usepackage{microtype}
\usepackage{amsfonts,amssymb,amsmath,amsthm}
\usepackage{nicefrac,xfrac}
\usepackage{mathtools} 
\mathtoolsset{showonlyrefs}
\usepackage[colorlinks=true, urlcolor=blue, linkcolor=blue, citecolor=blue]{hyperref}
\usepackage[
  backend=biber,
  style=alphabetic,
  sorting=anyt,
  maxnames=6,
  maxalphanames=6,
  giveninits=true,
  doi=true,
  isbn=false,
  url=false,
]{biblatex}
\DeclareFieldFormat{extraalpha}{#1}
\DeclareLabelalphaTemplate{
  \labelelement{
    \field[final]{shorthand}
    \field{label}
    \field[strwidth=3,strside=left,ifnames=1]{labelname}
    \field[strwidth=1,strside=left]{labelname}
  }
}

\numberwithin{equation}{section}
\newtheorem{theorem}{Theorem}[section]
\newtheorem{proposition}[theorem]{Proposition}
\newtheorem{corollary}[theorem]{Corollary}

\newtheorem{lemma}[theorem]{Lemma}

\theoremstyle{definition}
\newtheorem{remark}[theorem]{Remark}
\newtheorem{definition}[theorem]{Definition}
\newtheorem{question}[theorem]{Question}
\theoremstyle{remark}

\newcommand{\R}{\mathbb{R}}  
\newcommand{\C}{\mathbb{C}}  
\newcommand{\Sph}{\mathbb{S}}  
\DeclareMathAlphabet{\mathbbold}{U}{bbold}{m}{n}

\newcommand{\overbar}[1]{\mkern 1.5mu\overline{\mkern-1.5mu#1\mkern-1.5mu}\mkern 1.5mu}

\mathchardef\mhyphen="2D 
\DeclareMathOperator{\id}{id}

\DeclareMathOperator{\supp}{supp}

\DeclarePairedDelimiterX\set[2]{\{}{\}}{\,#1 \;\delimsize\vert\; #2\,}
\newcommand*{\abs}[1]{\left|#1\right|} 
\newcommand*{\norm}[1]{\left\|#1\right\|} 
\newcommand*{\brr}[1]{\left(#1\right)} 
\newcommand*{\brc}[1]{\left\{#1\right\}} 
\newcommand*{\brt}[1]{\left\langle #1\right\rangle } 
\mathchardef\mhyphen="2D 
\renewcommand*{\epsilon}{\varepsilon}
\renewcommand*{\phi}{\varphi}

\def\Xint#1{\mathchoice
{\XXint\displaystyle\textstyle{#1}}%
{\XXint\textstyle\scriptstyle{#1}}%
{\XXint\scriptstyle\scriptscriptstyle{#1}}%
{\XXint\scriptscriptstyle\scriptscriptstyle{#1}}%
\!\int}
\def\XXint#1#2#3{{\setbox0=\hbox{$#1{#2#3}{\int}$ }
\vcenter{\hbox{$#2#3$ }}\kern-.6\wd0}}

\def\dashint{\Xint-}
\makeatletter
\newcommand*{\oset}[3][0.45ex]{%
  \mathrel{\mathop{#3}\limits^{
    \vbox to#1{\kern-2\ex@
    \hbox{$\scriptstyle#2$}\vss}}}}
\makeatother

\title{Geometric bounds for Steklov and weighted Neumann eigenvalues on Euclidean domains}
\author{Denis Vinokurov}
\date{}

\begin{document}

\maketitle
\begin{abstract}
  We obtain sharp upper bounds for the first two nonzero Steklov eigenvalues among bounded domains in Euclidean spaces of dimension $d \geq 7$ under
  a natural normalization involving volume and boundary measure. These bounds are derived from a characterization of optimal domains and weights
  for the first two nonzero weighted Neumann eigenvalues. In dimensions $3 \leq d \leq 6$, we obtain upper bounds that are not sharp.

  We further establish strict upper bounds with sharp constants for all higher Steklov eigenvalues on simply connected planar domains
  with continuous boundary and finite boundary length. Beyond the second nonzero eigenvalue, these bounds were previously known only for smooth planar domains.
\end{abstract}

\section{Introduction and main results}
Let $\Omega \subset \R^d$ be a bounded open set (not necessarily connected), which we will refer to as a bounded domain, and let $d \geq 2$. Consider the Steklov eigenvalue problem
\begin{equation}\label{eq:Steklov-problem}
 \begin{cases}
   \Delta u = 0,
   \\ \partial_\nu u\left|_{\partial\Omega}\right. = \sigma u,
 \end{cases}
\end{equation}
where $\partial_\nu$ is the outward normal derivative along the boundary
of $\Omega$. When the trace operator $H^1(\Omega) \to L^2(\partial\Omega)$
is well-defined and compact (for instance, for Lipschitz domains) and $\Omega$ is connected, the spectrum of \eqref{eq:Steklov-problem}
is discrete and converges to infinity:
\begin{equation}\label{eq:Steklov-spectrum}
  0 = \sigma_0(\Omega) < \sigma_1(\Omega) \leq
  \sigma_2(\Omega) \leq \cdots\nearrow \infty,
\end{equation}
where the eigenvalues are repeated according to their multiplicity.

A natural question is which shapes maximize the $k$th Steklov eigenvalue within a given class of domains.
Since $\Omega \mapsto \sigma_k(\Omega)$ is a homogeneous functional, a suitable normalization of $\sigma_k$ is required.
Normalizations by powers of the volume and perimeter are among the natural choices.

For instance, Brock \cite{Brock:2001:first-stek-max-for-vol-norm} proved that
$\sigma_1(\Omega)|\Omega|^{1/d}$ is maximized uniquely by Euclidean balls. With another normalization,
Euclidean balls also uniquely maximize the quantity
$\sigma_1(\Omega)|\partial\Omega|^{1/(d-1)}$ both
among simply connected planar domains, by Weinstock~\cite{Weinstock:1954:first-steklov-simpl-dom}, and
among convex domains for $d\geq 3$, by
Bucur--Ferone--Nitsch--Trombetti~\cite{Bucur-Ferone-Nitsch-Trombetti:2021:weinstock-higher-dim}.
Moreover, for simply connected planar domains, Girouard and Polterovich \cite{Girouard-Polterovich:2010:payne-schiffer} proved that for $\sigma_2(\Omega)|\partial\Omega|$,
there is no maximizer; instead, the sharp value is approached by a sequence of domains converging to a disjoint union of two Euclidean balls (cf. Theorem~\ref{thm:max-lambda-2}).
We refer to \cite{Girouard-Polterovich:2017:steklov-survey, Colbois-Girouard-Gordon-Sher:2024:recent-develop-on-steklov} and references therein
for surveys of recent developments related to Steklov eigenvalues.

Among all normalizations given by the powers of $\abs{\partial\Omega}$ and $\abs{\Omega}$, the normalization
\begin{equation}\label{eq:shape-problem}
    \Omega \mapsto \sigma_k(\Omega)\abs{\partial\Omega}\abs{\Omega}^{\frac{2-d}{d}}
\end{equation}
appears to be the most geometric one:
when one varies Riemannian metrics within a conformal class together with boundary densities, the functional in~\eqref{eq:shape-problem} is the only one that admits
critical points, and they correspond to free boundary harmonic maps into Euclidean balls. See~\cite{Karpukhin-Metras:2022:higher-dim} for a more detailed discussion of different normalizations.

Colbois--El~Soufi--Girouard~\cite{Colbois-ElSoufi-Girouard:2011:isoperim-steklov} showed that in any dimension $d\geq 2$, both quantities \eqref{eq:shape-problem} and $\sigma_k(\Omega)|\partial\Omega|^{1/(d-1)}$
are uniformly bounded as $\Omega \subset \mathbb{R}^d$ ranges over bounded domains.
However, when $d \geq 3$, no sharp upper bounds are currently known even for $\sigma_1(\Omega)|\partial\Omega|^{1/(d-1)}$ in this general setting.
The main goal of this paper is to establish upper bounds for $\sigma_1$ and $\sigma_2$ with the normalization~\eqref{eq:shape-problem}, which are
sharp at least for $d \geq 7$.

An interesting phenomenon arises when one maximizes $\sigma_k(\Omega)|{\partial\Omega}|$ in dimension $d=2$.
Via homogenization and conformal invariance of the Steklov spectrum, Girouard--Karpukhin--Lagacé~\cite{Girouard-Karpukhin-Lagace:2021:continuity-of-eigenval} proved that the estimate
\begin{equation}
    \sigma_k(\Omega)|{\partial\Omega}| < 8\pi k =\sup_{g} \lambda_k(\Sph^2, g)\big|\Sph^2\big|_g
\end{equation}
is sharp. We will use similar homogenization ideas to obtain sharp upper bounds in higher dimensions for the following class of admissible $\Omega$,
which includes Lipschitz domains.
\begin{definition}
    A \emph{bounded} domain $\Omega$ is called \emph{admissible} if it satisfies the following three assumptions:
    \begin{itemize}
        \item $H^1(\Omega)\cap C^0(\overbar{\Omega})$ is dense in $H^1(\Omega)$;
        \item the embedding $H^1(\Omega) \to L^2(\Omega)$ is compact;
        \item $\overbar{\Omega} = \bigsqcup_{i} \overbar{\Omega}_i$, where $\Omega_i$ are the connected components of $\Omega$.
    \end{itemize}
\end{definition}
\begin{remark}
    Bounded domains with continuous boundary (that is, $\Omega$ can be locally represented as the epigraph of a continuous function) are admissible;
    see, for example, \cite[Theorem~1.1.6/2]{Mazja:1985:-sobolev-spaces} and \cite[Theorem~V.4.17]{Edmunds-Evans:1987:spec-theory-diff-operators}.
    In this case, $C^\infty(\overbar{\Omega})$ is dense in $H^1(\Omega)$.
\end{remark}
\begin{remark}
    The definition of admissibility is chosen mainly to ensure the applicability of the variational characterization of Corollary~\ref{cor:k-th-eigenval-as-inf}.
    The third property allows us to consider disjoint unions of connected domains.
\end{remark}
It is natural to view the Steklov problem as a particular case of an eigenvalue problem for measures, a perspective that has proved useful in its own
right \cite[see][]{Grigoryan-Netrusov-Yau:2004:eignval-of-ellipt-oper,Kokarev:2014:measure-eigenval,Girouard-Karpukhin-Lagace:2021:continuity-of-eigenval}.
Let $\mathcal{M}_+(\overbar{\Omega})$ be the space of nonnegative Radon measures on $\overbar{\Omega}$.
For $\mu \in \mathcal{M}_+(\overbar{\Omega})$, let us define variational Neumann eigenvalues $\lambda^N_k(\Omega, \mu) \in [0,\infty]$ as
\begin{equation}\label{eq:eigen-var-char-metric}
  \lambda^N_k(\Omega, \mu) := \inf_{V_{k+1}} \sup_{\phi \in V_{k+1}\setminus\brc{0}} \frac{\int_{\Omega} \abs{\mathrm{d}\phi}^2}{\int_{\overbar{\Omega}} \phi^2 d\mu},
\end{equation}
where $V_{k+1} \subset H^1(\Omega)\cap C^0(\overbar{\Omega})$ runs over all $(k+1)$-dimensional subspaces. We define
\begin{equation}
    \overbar{\lambda}^N_k(\Omega, \mu) = \mu(\overbar{\Omega}){\lambda}^N_k(\Omega, \mu),
\end{equation}
and we set $\lambda^N_k(\Omega, 0) = \infty$, $\overbar{\lambda}^N_k(\Omega, 0) = 0$ for convenience.

Let $\mathcal{M}_+^c(\overbar{\Omega}) \subset \mathcal{M}_+(\overbar{\Omega})$ be the subset of nonatomic (sometimes called continuous) measures.
Note that $\lambda^N_k(\Omega, \mu) < \infty$ provided that $L^2(\overbar{\Omega}, \mu)$ is at least $(k+1)$-dimensional. In particular, $\lambda^N_k(\Omega, \mu) < \infty$ as long as $0\neq \mu \in \mathcal{M}_+^c(\overbar{\Omega})$.

For a Lipschitz domain $\Omega$, one recovers the Steklov spectrum~\eqref{eq:Steklov-spectrum} by choosing $\mu = \mathcal{H}^{d-1}|_{\partial\Omega}$ to be the $(d-1)$-dimensional Hausdorff measure restricted to $\partial\Omega$:
\begin{equation}\label{eq:var-steklov-def}
    \overbar{\sigma}_k(\Omega) := \sigma_k(\Omega)|{\partial\Omega}| = \overbar{\lambda}^N_k(\Omega, \mathcal{H}^{d-1}|_{\partial\Omega}).
\end{equation}
It is then natural to extend the definition of the Steklov eigenvalues by formula~\eqref{eq:var-steklov-def} to the case of all
admissible domains $\Omega$ with $\abs{\partial \Omega} < \infty$. Note, however,
that the Steklov spectrum is not necessarily discrete when the boundary is not Lipschitz (see, for example, \cite{Nazarov-Taskinen:2020:noncomp-steklov}).

Consider
\begin{equation}
    \Lambda_k^N(\Omega)
    :=\sup_{\mu \in \mathcal{M}_+^c(\overbar{\Omega})}  \overbar{\lambda}^N_k(\Omega, \mu)
    \leq  C\abs{\Omega}^{\frac{d-2}{d}} k^{2/d},
\end{equation}
where the upper bound follows from~\cite[Remark~5.10]{Grigoryan-Netrusov-Yau:2004:eignval-of-ellipt-oper} and $C = C(d)$.
Clearly, if $\Omega$ runs over admissible domains, we have
\begin{equation}\label{ineq:stel-full-meas}
  \sup_{\abs{\partial \Omega} < \infty}\sigma_k(\Omega)|{\partial\Omega}|\abs{\Omega}^{\frac{2-d}{d}} \leq \sup_{\Omega} \Lambda_k^N(\Omega)\abs{\Omega}^{\frac{2-d}{d}}
  \leq C k^{2/d}.
\end{equation}
Nevertheless, the two suprema are closely related.
Let $L^1_+(\Omega) \subset L^1(\Omega)$ be the subset of a.e. nonnegative functions.
Combining \cite[Theorem~1.11]{Girouard-Karpukhin-Lagace:2021:continuity-of-eigenval} with \cite[Proposition~2.9]{Vinokurov:2025:higher-dim-harm-eigenval}, we have
\begin{proposition}
  \label{prop:steklov-approaches}
    Let $\Omega \subset \R^d$ be a bounded $C^1$-domain. Then there exists
    a family of $C^1$ domains $\Omega^\epsilon \subset \Omega$ such
    that $|\Omega^\epsilon|\to |\Omega|$ and
    \begin{equation}
        \sigma_k(\Omega^\epsilon)|{\partial\Omega^\epsilon}| \to \sup_{\mu \in L^1_+(\Omega)}  \overbar{\lambda}^N_k(\Omega, \mu).
    \end{equation}
\end{proposition}
We first explicitly calculate $\sup_{\Omega} \Lambda_k^N(\Omega)\abs{\Omega}^{\frac{2-d}{d}}$ for $k \in \brc{1,2}$, which can be viewed as a variant of a weighted
Neumann problem (cf. \cite{Bucur-Martinet-Oudet:2023:weighted-neumann}).
 Let $\mathbb{B}^d \subset \R^d$ be the unit ball centered at $0$, and let $\omega_{d}$ be the volume of $\mathbb{B}^d$. We also identify functions $f \in L^1$ with
 absolutely continuous measures $\mathrm{d}\mu = f(x)\mathrm{d} x$.
\begin{theorem}\label{thm:max-lambda-1}
    Let $\Omega \subset \R^d$ be an admissible domain such that $|\Omega| = |\mathbb{B}^d| = \omega_{d}$, and $d \geq 7$. Then for any
    $\mu \in \mathcal{M}^c_+(\overbar{\Omega})$, one has
    \begin{equation}
      \overbar{\lambda}^N_1(\Omega, \mu) \leq \overbar{\lambda}^N_1(\mathbb{B}^d, \tfrac{1}{\abs{x}^2}) = \tfrac{d(d-1)}{d-2}\omega_{d}.
    \end{equation}
    The bound is attained for $(\mathbb{B}^d, \frac{1}{\abs{x}^2})$. If equality holds, then, up to a translation,
    $\Omega = \mathbb{B}^d$ a.e. and $\mu$ is proportional to $\frac{1}{\abs{x}^2}$.
\end{theorem}
Since the spectrum of a disjoint union of domains is the union of the spectra, we have
$\overbar{\lambda}^N_2(\mathbb{B}^d \sqcup \mathbb{B}^d, \tfrac{1}{\abs{x}^2} \sqcup \tfrac{1}{\abs{x}^2}) = 2 \overbar{\lambda}^N_1(\mathbb{B}^d, \tfrac{1}{\abs{x}^2})$
in the next theorem.
\begin{theorem}\label{thm:max-lambda-2}
    Let $\Omega \subset \R^d$ be an admissible domain such that $|\Omega| = 2|\mathbb{B}^d| = 2 \omega_{d}$, and $d \geq 7$. Then for any
    $\mu \in \mathcal{M}^c_+(\overbar{\Omega})$, one has
    \begin{equation}
        \overbar{\lambda}^N_2(\Omega, \mu) \leq \overbar{\lambda}^N_2(\mathbb{B}^d \sqcup \mathbb{B}^d, \tfrac{1}{\abs{x}^2} \sqcup \tfrac{1}{\abs{x}^2})
        =\tfrac{d(d-1)}{d-2}(2\omega_{d}).
    \end{equation}
    The bound is attained for $(\mathbb{B}^d  \sqcup \mathbb{B}^d, \frac{1}{\abs{x}^2} \sqcup \frac{1}{\abs{x}^2})$.
    If equality holds, then $\Omega = \mathbb{B}^d  \sqcup \mathbb{B}^d$ a.e. and $\mu$ is proportional to $\frac{1}{\abs{x}^2} \sqcup \frac{1}{\abs{x}^2}$.
\end{theorem}
As an immediate corollary of \eqref{ineq:stel-full-meas}, the theorems above, and Proposition~\ref{prop:steklov-approaches} applied to balls, we obtain
\begin{theorem}\label{thm:steklov}
    Let $d\geq 7$ and $k\in\brc{1,2}$. Then, taking the supremum over all admissible domains $\Omega \subset \R^d$, we obtain
    \begin{equation}
      \sup_{\abs{\partial\Omega}<\infty}
      \sigma_k(\Omega)\abs{\partial\Omega}\abs{\Omega}^{\frac{2-d}{d}}
      =
      \sup_{\Omega}
      \Lambda_k^N(\Omega)\abs{\Omega}^{\frac{2-d}{d}}.
    \end{equation}
    In particular, for every admissible domain $\Omega \subset \R^d$ with $\abs{\partial\Omega}<\infty$,
    \begin{equation}
        \sigma_1(\Omega)\abs{\partial\Omega}\abs{\Omega}^{\frac{2-d}{d}} <  \tfrac{d(d-1)}{d-2}\omega_{d}^{2/d},
        \quad\quad
        \sigma_2(\Omega)\abs{\partial\Omega}\abs{\Omega}^{\frac{2-d}{d}} < \tfrac{d(d-1)}{d-2}(2\omega_{d})^{2/d}.
    \end{equation}
\end{theorem}
Note that it follows from the equality characterizations
in Theorems~\ref{thm:max-lambda-1} and~\ref{thm:max-lambda-2}
and \eqref{eq:var-steklov-def} that neither of these bounds
can be attained by Steklov eigenvalues.

We complement the estimates above by establishing strict upper bounds for all $\sigma_k(\Omega)$, $k\geq 2$, on simply connected domains in dimension 2.
As mentioned above, Weinstock \cite{Weinstock:1954:first-steklov-simpl-dom} proved that
\begin{equation}
  \overbar{\sigma}_1(\Omega) \leq 2\pi
\end{equation}
for simply connected domains with an analytic boundary, and the equality is achieved if and only if
$\Omega$ is a disk. Weinstock's result was later extended to Lipschitz domains by Girouard--Polterovich~\cite{Girouard-Polterovich:2010:low-neumann-steklov} and Freitas-Laugesen~\cite{Freitas-Laugesen:2020:steklov-neuman-via-robin}.
On the other hand, for higher Steklov eigenvalues on simply connected domains, the maximum is not achieved:
\begin{equation}\label{ineq:strict-simply-conn}
  \overbar{\sigma}_k(\Omega) <  \sup_{\abs{\partial \Omega} < \infty}\overbar{\sigma}_k(\Omega) = 2\pi k,  \quad k \geq 2.
\end{equation}
The case $k=2$ was proved by Girouard and Polterovich \cite{Girouard-Polterovich:2010:payne-schiffer} for Lipschitz domains,
while the remaining cases were established by Fraser--Schoen~\cite[Theorem~2.3]{Fraser-Schoen:2020:steklov-unions} and
the author~\cite[Remark~1.12]{Vinokurov:2025:sym-eigen-val-lms}, but only under the assumption that $\Omega$ is smooth. The regularity assumption can be weakened to
$C^{1,\alpha}$ domains by noticing that the Riemann mapping theorem yields a $C^1$-diffeomorphism
$f \colon \overbar{\mathbb{D}}^2 \to \overbar{\Omega}$ in that case via the Kellogg--Warschawski theorem.

However,
even for Lipschitz domains it remained open, to the best of the author's knowledge, whether inequality~\eqref{ineq:strict-simply-conn} is strict for all $k \geq 2$.
This question was raised in \cite[Section~4.4]{Grebenkov-Levitin-Polterovich:2026:Dirichlet-to-Neumann-map}.
We resolve this gap by proving the following result.
\begin{theorem}\label{thm:simply-connected-steklov}
  Let $\Omega \subset \R^2$ be a simply connected bounded domain with continuous boundary and $\abs{\partial\Omega} < \infty$. Then
  \begin{equation}
    \overbar{\sigma}_1(\Omega) \leq 2\pi
    \quad\text{and}\quad
     \overbar{\sigma}_k(\Omega) < 2\pi k \quad \forall k \geq 2.
  \end{equation}
  Moreover, $\overbar{\sigma}_1(\Omega) = 2\pi$ if and only if $\Omega$ is a disk.
\end{theorem}

\subsection{Strategy of the proof}

By the first variation formula \cite[see, for example,][Section~2.3]{Vinokurov:2025:higher-dim-harm-eigenval}, sufficiently regular critical measures of $\mu \mapsto \overbar{\lambda}^N_k(\Omega, \mu)$
are related to harmonic maps $u \colon \Omega \to \Sph^\infty$ given by the $k$th eigenfunctions ($\Delta u = \lambda_k u \mu$) satisfying
$\partial_{\nu} u|_{\partial\Omega} = 0$; that is, if $\mu$ is critical, it is proportional to $|\mathrm{d}u|^2$. Moreover, the existence results in \cite{Vinokurov:2025:higher-dim-harm-eigenval} can be generalized
to bounded Lipschitz domains, showing
that the supremum of $\overbar{\lambda}_k^N(\Omega,\mu)$ over absolutely continuous measures is attained by some $\mu \in L^1(\Omega)$ of the form
$\mu = \abs{\mathrm{d}u}^2$ for a harmonic map $u \in H^1(\Omega, \Sph^n)$. This suggests that the same may be true for
$\Lambda^N_k(\Omega)$, and it is natural to look for maximizing measures for $\Lambda_k^N(\Omega)$ among those induced by harmonic maps $u \colon \Omega \to \Sph^n$.

In the case $\Omega = \mathbb{B}^d$, consider the radial projection of $\mathbb{B}^d$ onto the unit sphere,
\begin{equation}
  u_0\colon \mathbb{B}^d \to \Sph^{ d-1}\quad \text{given by}\quad u_0(x) = \frac{x}{\abs{x}}.
\end{equation}
This is a basic example of a singular harmonic map. Note that $|\mathrm{d}u_0|^2 = ({d-1})/{|x|^2}$. In particular, $u_0 \in H^1(\mathbb{B}^d,\Sph^{ d-1})$ as long as $d \geq 3$. In
Lemma~\ref{lem:ball-index} below, we show that
\begin{equation}\label{ineq:lambda-energy}
        \overbar{\lambda}^N_1(\mathbb{B}^d, \abs{\mathrm{d}u_0}^2) \leq \int_{\mathbb{B}^d} \abs{\mathrm{d}u_0}^2,
\end{equation}
and the equality is achieved if and only if $d \geq 7$. At the same time, a short calculation yields
\begin{equation}
  \int_{\mathbb{B}^d} \abs{\mathrm{d}u_0}^2 = \tfrac{d(d-1)}{d-2}\omega_{d} \quad \forall d \geq 3.
\end{equation}
We will prove Theorems~\ref{thm:max-lambda-1} and~\ref{thm:max-lambda-2} by proving their extended versions below:
\begin{theorem}\label{thm:max-lambda-11}
    Let $\Omega \subset \R^d$ be an admissible domain such that $|\Omega| = |\mathbb{B}^d|$, and $d \geq 3$. Then for
    any $\mu \in \mathcal{M}^c_+(\overbar{\Omega})$, one has
    \begin{equation}
        \overbar{\lambda}^N_1(\Omega, \mu) \leq \int_{\mathbb{B}^d} \abs{\mathrm{d}u_0}^2.
    \end{equation}
    If the equality is achieved, then $\Omega = \mathbb{B}^d$ a.e. and $\mu$ is proportional to $\frac{1}{\abs{x}^2}$.
\end{theorem}
\begin{theorem}\label{thm:max-lambda-22}
    Let $\Omega \subset \R^d$ be an admissible domain such that $|\Omega| = 2|\mathbb{B}^d|$, and $d \geq 3$. Then for
    any $\mu \in \mathcal{M}^c_+(\overbar{\Omega})$, one has
    \begin{equation}
        \overbar{\lambda}^N_2(\Omega, \mu) \leq 2\int_{\mathbb{B}^d} \abs{\mathrm{d}u_0}^2.
    \end{equation}
    If the equality is achieved, then $\Omega = \mathbb{B}^d  \sqcup \mathbb{B}^d$ a.e. and $\mu$ is proportional to
    $\frac{1}{\abs{x}^2} \sqcup \frac{1}{\abs{x}^2}$.
\end{theorem}
As follows from~\eqref{ineq:lambda-energy}, the equalities are achieved only when $d \geq 7$. Then an extended version of
Theorem~\ref{thm:steklov} is the following:
 \begin{theorem}
     Let $\Omega \subset \R^d$ be an admissible domain with $\abs{\partial \Omega} < \infty$ and $d\geq 3$. Then one has
     \begin{equation}
         \sigma_1(\Omega)\abs{\partial\Omega}\abs{\Omega}^{\frac{2-d}{d}} <  \tfrac{d(d-1)}{d-2}\omega_{d}^{2/d},
         \quad\quad
         \sigma_2(\Omega)\abs{\partial\Omega}\abs{\Omega}^{\frac{2-d}{d}} < \tfrac{d(d-1)}{d-2}(2\omega_{d})^{2/d},
     \end{equation}
     and the inequalities are sharp if $d \geq 7$.
 \end{theorem}

The proof of Theorem~\ref{thm:max-lambda-11} follows the approach that Weinberger \cite{Weinberger:1956:first-neumann} used to show that the ball maximizes the first Neumann
eigenvalue among Euclidean domains. We use the coordinate functions $u_0^i$ as trial functions in the variational characterization of $\lambda_1$ and then apply
a topological argument to find a new origin $c \in \R^d$, so that all $u_0^i(\cdot-c)$ are orthogonal to constants. The proof of Theorem~\ref{thm:max-lambda-22}
is inspired by the works of Bucur--Henrot~\cite{Bucur-Henrot:2019:second-neumann}, Freitas--Laugesen~\cite{Freitas-Laugesen:2022:two-balls-neumann-max}, and
Girouard--Nadirashvili--Polterovich~\cite{Girouard-Nadirashvili-Polterovich:2009:planar-second-neumann},
which show that the second Neumann eigenvalue is maximized by the union of two balls (cf. also \cite{Nadirashvili:2002:second-eigen-sphere,Kim:2022:second-sphere-eigenval}).
We consider the composition of $u_0(\cdot - c)$ with a fold map across a hyperplane. By parametrizing the space of translations and hyperplanes as in
\cite{Freitas-Laugesen:2022:two-balls-neumann-max} and using \cite[Lemma~4.2]{Karpukhin-Stern:2024:new-character-of-conf-eigenval}, we construct
trial functions that are simultaneously orthogonal to constants and to the first eigenfunction, thereby avoiding a two-step orthogonalization procedure.

While Theorem~\ref{thm:max-lambda-1} describes all the pairs $(\Omega, \mu)$ maximizing the first weighted Neumann eigenvalue when $d \geq 7$, the situation
$3\leq d \leq 6$ remains unclear. As discussed above, \cite{Vinokurov:2025:higher-dim-harm-eigenval} suggests that $\Lambda_k^N(\Omega)$ is achieved by a measure
given by the energy density of a harmonic map into a sphere. In the context of closed manifolds, such harmonic maps are always smooth if $3\leq d \leq 6$
(see \cite[Corollary~1.3]{Vinokurov:2025:higher-dim-harm-eigenval}). It is natural to expect a similar result for domains if $\Omega$ is smooth.
\begin{question}
  Let $\Omega \subset \R^d$ range over admissible domains with $3\leq d\leq6$. Does the following equality hold?
  \begin{equation}
    \sup_{|\Omega| = |\mathbb{B}^d|} \Lambda_1^N(\Omega) = \Lambda_1^N(\mathbb{B}^d).
  \end{equation}
\end{question}
One may first attempt to find a (presumably smooth) harmonic map $u \colon \mathbb{B}^d \to \Sph^n$ realizing $\Lambda_1^N(\mathbb{B}^d)$, and then
its structure may suggest how to adapt the topological argument for the upper bound.

The proof of Theorem~\ref{thm:simply-connected-steklov} consists of two ingredients. First, we note that when the boundary is a rectifiable Jordan curve,
the Riesz--Privalov theorem produces a conformal homeomorphism $f \colon \overbar{\mathbb{D}}^2 \to \overbar{\Omega}$ that is absolutely
continuous on the boundary. In particular, we reduce the problem to estimating the weighted Steklov eigenvalues $\overbar{\sigma}_k(\mathbb{D}^2, \rho)$ when
$\rho \in L^1(\partial \mathbb{D}^2)$. The second ingredient is that all maximizing
densities $\rho \in L^1(\partial \mathbb{D}^2)$ are actually induced by smooth free boundary harmonic maps $u \colon \mathbb{D}^2 \to \mathbb{B}^n$
(see Lemma~\ref{lem:qualitative-steklov-stability}).

The remainder of the paper is devoted to the proofs of \eqref{ineq:lambda-energy} and Theorems~\ref{thm:max-lambda-11},~\ref{thm:max-lambda-22},
and~\ref{thm:simply-connected-steklov}.

\subsection{Acknowledgments}
This work forms part of the author's PhD thesis, carried out under the supervision of Mikhail Karpukhin and Iosif Polterovich.
The author is grateful to both for their guidance and many fruitful discussions, and especially to Mikhail Karpukhin for suggesting the problem.
The author was partially supported by an ISM scholarship. Part of this work was completed during the programme
\emph{Geometric spectral theory and applications} at the Isaac Newton Institute for Mathematical Sciences, Cambridge,
whose support and hospitality are gratefully acknowledged. This programme was supported by EPSRC grant EP/Z000580/1.

\section{Preliminaries}\label{sec:prelim}
\begin{lemma}\label{lem:unstable-ring}
  Let $\Omega \subset M$ be a domain of a Riemannian manifold $(M,g)$ of dimension $d \geq 2$, and let $0\neq\mu \in \mathcal{M}_+^{c}(\overbar{\Omega})$ with $\lambda^N_k(\Omega, \mu) =1$ for some $k > 0$.
  Then every point $p \in \overbar{\Omega}$ has a neighborhood $U \subset M$ such that for all
    $\phi \in H^1(\Omega)\cap C^0(\overbar{\Omega})$ with $\supp \phi \subset U \cap\overbar{\Omega}$,
    one has
    \begin{equation}\label{eq:unstable-ring}
      \int_{\overbar{\Omega}} \phi^2 \mathrm{d}\mu \leq \int_{\Omega} \abs{\mathrm{d}\phi}^2.
    \end{equation}
\end{lemma}
\begin{proof}
  Note that it suffices to prove \eqref{eq:unstable-ring} only for those $\phi$ with $\supp \phi \subset (U\setminus\brc{p}) \cap\overbar{\Omega}$, since
  discrete sets have zero capacity and $\mu(\brc{p}) = 0$.

  Arguing by contradiction, we can find a sequence of functions $\phi_i \in H^1(\Omega)\cap C^0(\overbar{\Omega})$ with disjoint supports
  $\supp \phi_i \subset U_i \cap\overbar{\Omega}$, where $U_i := B_{r_i}(p)\setminus\overline{B_{r_{i+1}}(p)}$ and $r_i \searrow 0$, such that
  \begin{equation}
    \int_{\Omega} \abs{\mathrm{d}\phi_i}^2 - \int_{\overbar{\Omega}} \phi_i^2 \mathrm{d}\mu < 0.
  \end{equation}
  Choosing $k+1$ such functions yields a contradiction with the variational characterization of $\lambda^N_k(\Omega, \mu)=1$.
\end{proof}
\begin{proposition}
    Let $\Omega \subset M$ be a bounded domain of a Riemannian manifold $(M,g)$, and let $0\neq\mu \in \mathcal{M}_+^{c}(\overbar{\Omega})$. If
    $H^1(\Omega)\cap C^0(\overbar{\Omega})$ is dense in $H^1(\Omega)$ and $\lambda^N_k(\Omega, \mu) \neq 0$ for some $k > 0$, then the measure
    $\mu$ induces a continuous bilinear form on $H^1(\Omega)$, that is, $\mu \in \mathfrak{Bil}[H^1(\Omega)]$.
\end{proposition}
\begin{proof}
    By Lemma~\ref{lem:unstable-ring}, every point $p \in \overbar{\Omega}$ has a neighborhood $U$ such that for all
    $\phi \in H^1(\Omega)\cap C^0(\overbar{\Omega})$ with $\supp \phi \subset U \cap\overbar{\Omega}$, one has
    \begin{equation}
        \lambda^N_k(\Omega, \mu)\int_{\overbar{\Omega}} \phi^2 \mathrm{d}\mu \leq \int_{\Omega} \abs{\mathrm{d}\phi}^2.
    \end{equation}
    Then a finite partition of unity argument with $\sum_i \eta_i^2 = 1$,
    $\supp \eta_i \subset U_i$, and $\overbar{\Omega} \subset \bigcup_i U_i$, implies
    that for all $\phi \in H^1(\Omega)\cap C^0(\overbar{\Omega})$,
    \begin{equation}
        \lambda^N_k(\Omega, \mu)\int_{\overbar{\Omega}} \phi^2 \mathrm{d}\mu \leq \int_{\Omega} \abs{\mathrm{d}\phi}^2
        + \frac{1}{2} \sum_i \int_{\Omega} \brt{\mathrm{d}\eta_i^2, \mathrm{d}\phi^2} + \sum_i \int_{\Omega} \phi^2\abs{\mathrm{d}\eta_i}^2,
    \end{equation}
    where the middle sum vanishes since $\sum_i \mathrm{d} \eta_i^2 = 0$. Hence, there exists a constant $C > 0$ such that
    \begin{equation}
        \lambda^N_k(\Omega, \mu)\int_{\overbar{\Omega}} \phi^2 \mathrm{d}\mu \leq \int_{\Omega} \abs{\mathrm{d}\phi}^2 + C\int_{\Omega} \phi^2.
    \end{equation}
\end{proof}
Thus, the identity map $H^1(\Omega)\cap C^0(\overbar{\Omega})\to L^2(\overbar{\Omega},\mu)$ induces a continuous linear map $H^1(\Omega) \to L^2(\overbar{\Omega},\mu)$. Integration with respect to $\mu$ will be understood via this map.
We will need the following abstract version of the Poincaré inequality:
\begin{lemma}[{\cite[Lemma~4.1.3]{Ziemer:1989:weakly-diff-func}}]
  \label{lem:poincare-ineq}
  Let $X_0$ be a normed space with norm $\norm{\cdot}_0$, and let $X \subset X_0$ be a Banach space with
  norm $\norm{\cdot}$. Assume that $\norm{\cdot}= \norm{\cdot}_0 + \norm{\cdot}_1$ for some seminorm $\norm{\cdot}_1$, and that the embedding
  $X \hookrightarrow X_0$ is compact. Let $Y = \set{x \in X}{\norm{x}_1 = 0}$. Then there exists a constant $C > 0$ such that
  for any projection $P \colon X \to Y$, one has
  \begin{equation}
    \norm{x - Px}_0 \leq C \norm{P} \norm{x}_1 \quad \forall x \in X.
  \end{equation}
\end{lemma}
\begin{corollary}
  Let $\Omega$ be an admissible connected domain and $0\neq\mu \in \mathcal{M}_+^{c}(\overbar{\Omega})$. If $k > 0$ and  $\lambda^N_k(\Omega, \mu) \neq 0$,
  then the norm
  \begin{equation}
        \norm{\phi}^2_* := \int_{\Omega} \abs{\mathrm{d}\phi}^2 + \int_{\overbar{\Omega}} \phi^2 \mathrm{d}\mu
  \end{equation}
  is an equivalent norm on $H^1(\Omega)$.
\end{corollary}
\begin{proof}
  By the previous proposition, we have $\mu \in \mathfrak{Bil}[H^1(\Omega)]$. So, it remains to prove that $\norm{\cdot}_*$ is bounded from below. Consider
  the projection to constant functions $P\colon H^1(\Omega) \to \R\subset H^1(\Omega)$ given by
  \begin{equation}
    P\phi = \dashint_{\overbar{\Omega}} \phi \mathrm{d}\mu
  \end{equation}
  and apply Lemma~\ref{lem:poincare-ineq}. Since the embedding $H^1(\Omega) \to L^2(\Omega)$ is compact, we obtain
  \begin{equation}
    \norm{\phi}_{L^2} \leq \norm{P\phi}_{L^2} + \norm{\phi - P\phi}_{L^2}
    \leq C\brr{\norm{\phi}_{L^2(\mu)} + \norm{\mathrm{d}\phi}_{L^2}}.
  \end{equation}
\end{proof}
\begin{corollary}\label{cor:k-th-eigenval-as-inf}
    Let $\Omega$ be an admissible domain and $0\neq\mu \in \mathcal{M}_+^{c}(\overbar{\Omega})$. If $k > 0$ and  $\lambda^N_k(\Omega, \mu) \neq 0$, then
    there exists a subspace $V \subset H^1(\Omega)$ such that
    $1 \in V$, $\dim V \leq k$, and
    \begin{equation}\label{eq:var-char}
        \lambda^N_k(\Omega, \mu) = \inf \set*{\frac{\int_{\Omega} \abs{\mathrm{d}\phi}^2}{\int_{\overbar{\Omega}} \phi^2 \mathrm{d}\mu}}{\phi \in H^1(\Omega)\setminus\brc{0},\ \int_{\overbar{\Omega}} \phi \psi \mathrm{d}\mu = 0 \ \forall \psi \in V}.
    \end{equation}
    In fact, $V = \bigoplus_{\lambda < \lambda_k^N} V_{\lambda}$, where $V_{\lambda} \subset H^1(\Omega)$
    is the space of eigenfunctions associated with $\lambda$ ($\Delta u = \lambda u \mu$).
\end{corollary}
\begin{proof}
    Recall that $\overbar{\Omega} = \bigsqcup_{i} \overbar{\Omega}_i$, where $\Omega_i$ are the connected components of $\Omega$.
    Then $\bigcup_k\brc{\lambda^N_k(\Omega, \mu)} = \bigcup_{k,i}\brc{\lambda^N_k(\Omega_i, \mu)}$, and it suffices to prove~\eqref{eq:var-char} for each connected component $\Omega_i$ such that $\mu(\overbar{\Omega}_i) \neq 0$. Hence, we assume that $\Omega$ is connected.

    By the previous corollary, we have $\mu \in \mathfrak{Bil}[H^1(\Omega)]$ and an equivalent norm
    \begin{equation}
        \norm{\phi}^2_* = \int_{\Omega} \abs{\mathrm{d}\phi}^2 + \int_{\overbar{\Omega}} \phi^2 \mathrm{d}\mu \quad \text{on } H^1(\Omega).
    \end{equation}
    Then the numbers $\frac{1}{1 + \lambda^N_i(\Omega, \mu)}$ form the top part of the spectrum of the operator
    \begin{equation}\label{eq:mu-operator}
        T_\mu := (\mu+\Delta)^{-1}\mu
    \end{equation}
    on $H^1(\Omega)$, induced by $\mu\in\mathfrak{Bil}[H^1(\Omega)]$ with respect to the inner product associated with $\norm{\cdot}_*^2$:
    \begin{equation}
        \frac{1}{1 + \lambda^N_{k-1}(\Omega, \mu)}
    = \sup_{V_{k} \subset H^1(\Omega)} \inf_{\phi \in V_{k}\setminus\brc{0}} \frac{\int \phi^2\mathrm{d}\mu}{\norm{\phi}^2_*}
    = \sup_{V_{k} \subset H^1(\Omega)} \inf_{\phi \in V_{k}\setminus\brc{0}} \frac{\brt{T_\mu\phi,\phi}_*}{\norm{\phi}^2_*}.
    \end{equation}
    Thus, the variational characterization~\eqref{eq:var-char} follows from the analogous characterization for bounded self-adjoint operators on a Hilbert space.
\end{proof}

\section{Proofs of the main results}\label{sec:proofs}
\subsection{The ball maximizes \texorpdfstring{$\overbar{\lambda}_1^N(\Omega, \mu)$}{Lg}}
\begin{lemma}\label{lem:ball-index}
    Let $d \geq 3$. One has
    $\lambda_1^N(\mathbb{B}^d,\tfrac{1}{\abs{x}^2}) = \min \brc{d-1, (\frac{d-2}{2})^2}$.
    That is, $\lambda_1^N(\mathbb{B}^d,\tfrac{1}{\abs{x}^2}) = d-1$ when $d \geq 7$,
    and $\lambda_1^N(\mathbb{B}^d,\tfrac{1}{\abs{x}^2}) < d-1$ otherwise. Moreover, the value $(\frac{d-2}{2})^2$ is the bottom of the essential spectrum.
\end{lemma}
\begin{proof}
    By the Hardy inequality, the quadratic form $\mathfrak{q}[\varphi] := \int_{\mathbb{B}^d}  \frac{\varphi^2}{\abs{x}^2}$ is continuous as a form on $H^1(\mathbb{B}^d)$ when $d \geq 3$.
    The decomposition by normalized spherical harmonics yields
    \begin{equation}\label{eq:ball-decomp}
        H^1(\mathbb{B}^d) = H^1(\mathbb{B}^d\setminus\brc{0})
        = \overline{\bigoplus_i H^1((0,1], r^{d-1}\mathrm{d}r) Y_i},
    \end{equation}
    where $\brc{Y_i } \subset C^\infty(\Sph^{d-1})$ is an orthonormal basis of spherical harmonics,
    $\Delta_{\Sph^{d-1}} Y_i = \nu_i Y_i$, and $\nu_i \in \set{\ell(d-2+\ell)}{\ell = 0, 1,2,\cdots}$ counting multiplicities.
    On $\psi(x) = \phi(r)Y_i$ with $\phi(r) \in H^1((0,1], r^{d-1}\mathrm{d}r)$, we thus have
    \begin{equation}\label{eq:shirt-spectr}
        \frac{\int_{\mathbb{B}^d} \abs{\mathrm{d}\psi}^2}{\int_{\mathbb{B}^d} \psi^2\frac{\mathrm{d}x}{|x|^2} }
        = \frac{\int_{0}^1 \phi'(r)^2r^{d-1}\mathrm{d}r}{\int_{0}^1 \phi(r)^2 r^{d-3}\mathrm{d}r}+\nu_i.
    \end{equation}
    Let $L$ be the self-adjoint operator on $L^2((0,1), r^{d-3}\mathrm{d}r)$ associated with the radial quadratic form. Its differential expression
    is $L\varphi := -r^{3-d}(r^{d-1}\varphi(r)')'$
    with the boundary condition $\varphi'(1) = 0$. On the $L^2(\mathbb{B}^d, \frac{1}{|x|^2})$-subspace corresponding to $Y_i$,
    the Dirichlet form is associated with $L + \nu_i$. Since the decomposition~\eqref{eq:ball-decomp} is orthogonal with respect
    to both the Dirichlet quadratic form and the $\frac{1}{|x|^2}$-volume form, the full operator is the orthogonal direct sum
    $\bigoplus_i \brr{L + \nu_i}$ and its spectrum is
    \begin{equation}\label{eq:full-spec}
        \overline{\bigcup_i \brc{\sigma(L) + \nu_i}} = \bigcup_i \brc{\sigma(L) + \nu_i}.
    \end{equation}
    The general solution of $L\varphi =\lambda \varphi$ has the form
    $
        \varphi(r) = c_1 r^{\beta_+}+ c_2r^{\beta_{-}},
    $
    where
    \begin{equation}\label{eq:eigen-func-roots}
        \beta_{\pm} = -\frac{d-2}{2} \pm \sqrt{\brr{\frac{d-2}{2}}^2 - \lambda}.
    \end{equation}
    When $\lambda=\brr{\frac{d-2}{2}}^2$, the second linearly
    independent solution is
    $r^{-\frac{d-2}{2}}\ln r$.
    If $\lambda<\brr{\frac{d-2}{2}}^2$, only
    $r^{\beta_+}$ among the two solutions belongs to
    $L^2((0,1),r^{d-3}\mathrm{d}r)$. At the threshold neither
    solution belongs to this space, while for
    $\lambda>\brr{\frac{d-2}{2}}^2$ the two solutions have modulus
    comparable to $r^{-\frac{d-2}{2}}$ and are again not in
    $L^2((0,1),r^{d-3}\mathrm{d}r)$.
    Thus, any eigenfunction of $L$ is of the form
    $\varphi(r)=c_1r^{\beta_+}$. Finally, the condition $\varphi'(1) = 0$ yields $\beta_+ = 0$, and hence $\varphi$ is constant. Thus,
    \begin{equation}
        \sigma(L) = \brc{0} \sqcup \sigma_{ess}(L).
    \end{equation}
    The bottom of the essential spectrum can be computed as follows \cite[see][Theorem~14.9c]{Weidmann:1987:ode-spectral-theory}:
    \begin{equation}
        \inf \sigma_{ess}(L) = \sup \set*{\lambda \in \R}{\exists \varphi \colon (L-\lambda)\varphi = 0 \text{ and } \phi^{-1}(0) \cap (0,1) \text{ is finite} }.
    \end{equation}
    Hence, $\inf \sigma_{ess} = (\tfrac{d-2}{2})^2$ by~\eqref{eq:eigen-func-roots}; cf. also~\cite[Lemma~1.3]{Schoen-Uhlenbeck:1984:min-harm-maps}.

    Since $0$ is the only discrete eigenvalue of $L$, while $\nu_0 = 0$ and $\nu_1 = d-1$, the spectrum structure~\eqref{eq:full-spec} and
    the variational characterization of $\lambda_1^N$ give
    \begin{equation}
      \lambda_1^N(\mathbb{B}^d,\tfrac{1}{\abs{x}^2}) = \min\brc{d-1, \brr{\frac{d-2}{2}}^2}.
    \end{equation}
\end{proof}

\subsubsection{Proof of Theorem~\ref{thm:max-lambda-11}}
    We may assume that $\overbar{\lambda}_1^N(\Omega,\mu)>0$, since otherwise the claimed inequality is immediate.
    One may assume that $\Omega \subset B$, where $B= \mathbb{B}_R^d(0)$ for some $R$. We define a map $\Phi \colon \overbar{B} \to \R^d$ by
    \begin{equation}
        \Phi(c) = \dashint_{\overbar{\Omega}} \frac{c-x}{\abs{c-x}} \mathrm{d}\mu(x).
    \end{equation}
    The map $\Phi$ is easily seen to be continuous -- either by the dominated convergence theorem or by the fact that $\mu \in \mathfrak{Bil}[H^1]$ if $\lambda_1^N(\Omega,\mu) \neq 0$.
    When $c \in \partial B$ and $x \in \Omega$, we see that $\brt{c-x, c} > 0$ and hence $\brt{\Phi(c), c} > 0$, which implies that $\Phi \colon{\partial B \to \R^{d}\setminus \brc{0}}$ is homotopic to the identity map. In particular, $\deg \Phi|_{\partial B \to \R^{d}\setminus \brc{0}} \neq 0$, and there exists $c \in B$ such that $\Phi(c) = 0$.
    Otherwise, $\Phi$ would be homotopic to a constant map with $\deg \Phi = 0$.
    Therefore, we may assume that $\Omega$ is centered in such a way that
    \begin{equation}
        \int_{\overbar{\Omega}} \frac{x}{\abs{x}} \mathrm{d}\mu = 0.
    \end{equation}
    That is, all the coordinate functions of ${x}/{\abs{x}}$ are orthogonal to constants. Recall that
    $|\Omega| = |\mathbb{B}^d|$. By Corollary~\ref{cor:k-th-eigenval-as-inf}, we obtain
    \begin{equation}
      \lambda^N_1(\Omega, \mu) \int_{\overbar{\Omega}}\frac{x_i^2}{\abs{x}^2}\mathrm{d}\mu  \leq \int_{\Omega}\abs{\mathrm{d}\brr{\frac{x_i}{\abs{x}}}}^2.
    \end{equation}
    Then summing over $i$ yields
    \begin{equation}\label{ineq:coord-var-char}
      \begin{aligned}
        \overbar{\lambda}^N_1(\Omega, \mu) &\leq \int_{\Omega}\abs{\mathrm{d}\brr{\frac{x}{\abs{x}}}}^2 = \int_{\Omega}\frac{d-1}{\abs{x}^2} &
        \\ &=
        \int_{\Omega \cap \mathbb{B}^d}\frac{d-1}{\abs{x}^2}
        +  \int_{\Omega \setminus \mathbb{B}^d}\frac{d-1}{\abs{x}^2}
        \\ &\leq  \int_{\mathbb{B}^d}\frac{d-1}{\abs{x}^2},
      \end{aligned}
    \end{equation}
    since $\frac{1}{\abs{x}^2} \leq 1$ on $\Omega \setminus \mathbb{B}^d$, $1 \leq \frac{1}{\abs{x}^2}$ on $\mathbb{B}^d\setminus \Omega$,
    and $|\Omega \setminus \mathbb{B}^d| = |\mathbb{B}^d\setminus \Omega|$.

    The equality occurs only if  $|\Omega \setminus \mathbb{B}^d| = |\mathbb{B}^d\setminus \Omega| = 0$ and the coordinate functions of
    $u_0\colon x \mapsto {x}/{|x|}$ are the eigenfunctions, that is $\Delta u^i_0 = \lambda_1 u^i_0\mu$, in which case $\mu$ is proportional to
    $\abs{\mathrm{d}u_0}^2$ since $\abs{u_0}^2 = 1$.

\subsection{The two balls maximize \texorpdfstring{$\overbar{\lambda}_2^N(\Omega, \mu)$}{Lg}}

Let $p \in \R^d\setminus\brc{0}$ and $R_p$ be the reflection
\begin{equation}
    R_p(y) = y - 2 \brt{y,\frac{p}{\abs{p}}}\frac{p}{\abs{p}}.
\end{equation}
Analogously to \cite{Freitas-Laugesen:2022:two-balls-neumann-max}, we
    define $H_{p,t} := \set{y \in \R^d}{ \brt{y,p} < t\abs{p}}$, where $p \neq 0$ and $t \geq 0$. Let $R_{p,t}(y) = y + 2\brr{t - \brt{y,\frac{p}{\abs{p}}}}\frac{p}{\abs{p}}$ be the reflection in the hyperplane $\partial H_{p,t}$. The ``fold map''
    is defined as
    \begin{equation}
        F_{p,t} := \begin{cases}
        \id & \quad\text{on}\quad H_{p,t}
        \\ R_{p,t} &\quad\text{on}\quad \R^d\setminus H_{p,t}.
        \end{cases}
    \end{equation}

\subsubsection{Proof of Theorem~\ref{thm:max-lambda-22}}
    We may assume that $\overbar{\lambda}_2^N(\Omega,\mu)>0$, since otherwise the claimed inequality is immediate.
    Again, as in the proof of Theorem~\ref{thm:max-lambda-11}, we assume that $\Omega \subset B$ for some ball $B = \mathbb{B}_R^d(0)$.
    Let $V = \mathrm{span} \brc{1, \phi_1}$ from Corollary~\ref{cor:k-th-eigenval-as-inf} ($\phi_1$ may be $0$) and
    consider a continuous map $\Phi\colon \overbar{B}\times \overbar{B}
    \to \R^d \times \R^d$ given by
    \begin{equation}\label{eq:def-Phi}
        \Phi(c, p) = \brr{\dashint_{\overbar{\Omega}} \frac{c-F_{p,R-\abs{p}}(x)}{\abs{c-F_{p,R-\abs{p}}(x)}} \mathrm{d}\mu(x), \dashint_{\overbar{\Omega}} \frac{c-F_{p,R-\abs{p}}(x)}{\abs{c-F_{p,R-\abs{p}}(x)}} \phi_1(x) \mathrm{d}\mu(x)}.
    \end{equation}
    Note that $\Phi(c, 0)$ is well defined since
    $\Omega \subset H_{p,R}$, and $F_{p,R}|_{\overbar{\Omega}} = \id$ does not depend on $p$. The map $\Phi = (\Phi',\Phi'')$ has the following two properties:
    \begin{itemize}
        \item $\brt{\Phi'(c,p), c} > 0$ when $c \in \R^d\setminus B$, since $F_{p,t}(B) \subset B$ and $\brt{c-F_{p,t}(x), c} > 0$;
        \item $\Phi(R_p (c),-p) = (R_p\times R_p)(\Phi(c,p))$ when $p \in \partial B$ since $F_{-p,0} = R_p \circ F_{p,0}$.
    \end{itemize}

    We aim to find a pair $(c, p) \in \overbar{B}\times \overbar{B}$ with $\Phi(c, p) = 0$.
  Suppose that no such pair exists. Then $\Phi(\overbar{B}\times \overbar{B}) \subset \R^{2d}\setminus \brc{0}$, and $\Phi$ is homotopic to a constant map. We will also prove $\deg \Phi|_{\partial (B\times B) \to \R^{2d}\setminus \brc{0}} \neq 0$. Set
    \begin{equation}
        \Phi_t(c,p) := \Phi\brr{\frac{c}{1- t\tfrac{\abs{c}}{R}},p}.
    \end{equation}
    Using the first property above and the fact that $\Phi(\overbar{B}\times \overbar{B}) \subset \R^{2d}\setminus \brc{0}$, we see that this is a homotopy in $\R^{2d}\setminus\brc{0}$ between $\Phi = \Phi_0$ and $\Phi_1$,
    where the latter has almost the same formula as $\Phi$ with the only difference that all the $F_{p,R-|p|}$ are multiplied by $(1- \abs{c}/{R})$.
    It is easy to see that $\Phi_1$ satisfies the second property (even for $(c,p) \in \partial B \times B$, as $\Phi_1(c,p) = (c/R,0)$ in this case) and therefore has
    a nonzero degree by \cite[Lemma~4.2]{Karpukhin-Stern:2024:new-character-of-conf-eigenval}, with the natural identification $\Sph^{2d-1} \approx \partial (B\times B)$, $(a,b) \mapsto \frac{R\cdot(a,b)}{\max\brc{|a|,|b|}}$. Thus, we obtain a contradiction.

    Therefore, we can choose coordinates so that $c=0$ and choose $\mathcal{R} = R_{p, R-\abs{p}}$, $H = H_{p, R-\abs{p}}$,
    and $F = F_{p, R-\abs{p}}$ so that $\tfrac{F}{\abs{F}} \bot_{\mu} V$. By Corollary~\ref{cor:k-th-eigenval-as-inf}, similarly to~\eqref{ineq:coord-var-char}, we have
    \begin{align}
        \overbar{\lambda}^N_2(\Omega, \mu)
        &\leq \int_{\Omega}\abs{\mathrm{d}\brr{\frac{F}{\abs{F}}}}^2
        =\int_{\Omega\cap H}\frac{d-1}{\abs{x}^2}
        + \int_{\mathcal{R}(\Omega\setminus H)}\frac{d-1}{\abs{x}^2} \\
        & \leq 2\int_{\mathbb{B}^d}\frac{d-1}{\abs{x}^2},
    \end{align}
    where the last inequality follows from \cite[Lemma~4.1]{Freitas-Laugesen:2022:two-balls-neumann-max}, together with its sharpness conditions.

\subsection{Upper bounds for planar domains}
  For an admissible domain with a density $\rho \in L^1_+(\partial \Omega)$ with respect to the $(d-1)$-dimensional Hausdorff measure, we define
  weighted Steklov eigenvalues $\sigma_k(\Omega,\rho)$ by
  \begin{equation}
    \sigma_k(\Omega,\rho):= \lambda_k^N(\Omega,\rho\mathrm{d}\mathcal{H}^{d-1}|_{\partial \Omega})
    \quad\text{and}\quad
    \overbar{\sigma}_k(\Omega,\rho):=\sigma_k(\Omega,\rho)\int_{\partial\Omega}\rho\,\mathrm{d}\mathcal{H}^{d-1}.
  \end{equation}

  It is known \cite[see, for example,][]{Petrides:2019:max-steklov-eigenval-on-surfaces, Vinokurov:2025:sym-eigen-val-lms} that
  the suprema of the Steklov eigenvalues within conformal classes on Riemannian surfaces are achieved (modulo bubbling) on free boundary harmonic maps into balls.
  The purpose of the following lemma is to show that all absolutely continuous maximizers are obtained in such a way.

  \begin{lemma}\label{lem:qualitative-steklov-stability}
    Let $\Omega$ be a smooth connected Riemannian surface with nonempty boundary, $k \ge 1$, and let $\rho \in L^1_+(\partial \Omega)$ be such that
    \begin{equation}
      \overbar{\sigma}_k(\Omega,\rho) = \sup_{\tilde{\rho} \in L^1_+(\partial \Omega)}\overbar{\sigma}_k(\Omega,\tilde{\rho}) =: \mathfrak{S}_k(\Omega).
    \end{equation}
    Then there exists a free boundary harmonic map $u \in C^\infty(\overbar{\Omega}, \overbar{\mathbb{B}}^n)$ given by the
    $\sigma_k$-eigenfunctions. In particular, $\sigma_k(\Omega,\rho)\rho = \brt{u, \partial_\nu u} \in C^\infty(\partial\Omega)$.
  \end{lemma}
  \begin{proof}
    The proof of this lemma is a straightforward adaptation of~\cite[Theorem~1.1]{Vinokurov:2026:p-harm-and-conf-class-opt-gafa} to the Steklov case.

    We already know that $\mathfrak{S}_k(\Omega)$ is achieved on a density induced by a free boundary harmonic map provided that no bubbling occurs
    (see, for example, \cite[Theorem~1.5]{Vinokurov:2025:sym-eigen-val-lms}). The proof of \cite[Theorem~1.5]{Vinokurov:2025:sym-eigen-val-lms} follows essentially the following scheme.
    First, we construct a ``Palais-Smale'' maximizing sequence of densities $\rho_\epsilon \in L^\infty(\partial\Omega)$,
    $\overbar{\sigma}_k(\Omega,\rho_\epsilon) \to \mathfrak{S}_k(\Omega)$ together with the eigenmaps $u_\epsilon \in H^1(\Omega, \overbar{\mathbb{B}}^n)$ by \cite[Proposition~3.1]{Vinokurov:2025:sym-eigen-val-lms}
    and then use \cite[Lemmas~3.6 and~3.8]{Vinokurov:2025:sym-eigen-val-lms} to prove that for an appropriate subsequence,
    $u_\epsilon$ converges weakly in $H^1(\Omega)$ to a smooth free boundary harmonic map $u_*$ into the ball, while the measures
    $\sigma_k(\Omega,\rho_\epsilon) \rho_\epsilon \mathrm{d}\mathcal{H}^{1}|_{\partial \Omega}$
    converge weak-$*$ to $\brt{u_*, \partial_\nu u_*}\mathrm{d}\mathcal{H}^{1}|_{\partial \Omega}$.

    Thus, for the proof of the present lemma (exactly as for \cite[Theorem~1.1]{Vinokurov:2026:p-harm-and-conf-class-opt-gafa}),
    it would be sufficient to show that a ``Palais-Smale'' sequence $\brc{\rho_\epsilon}$ can be chosen so that $\rho_\epsilon \to \rho$ in $L^1(\partial\Omega)$.
    We outline the main steps in this direction.

    Rescale $\rho$ so that $\int_{\partial\Omega}\rho \mathrm{d}\mathcal{H}^{1}= 1$, and consider truncated functions $\tilde{\rho}_{\epsilon} := \min \brc{\rho, \epsilon^{-1}} \in L^\infty(\partial\Omega)$. Clearly,
    $\tilde{\rho}_{\epsilon} \to \rho$ in $L^1(\partial\Omega)$.
    By the (variational) definition of $\sigma_k$, $\tilde{\rho}_{\epsilon} \leq \rho$ implies $\sigma_k(\Omega, \rho) \leq \sigma_k(\Omega, \tilde{\rho}_\epsilon)$.
    Since $\overbar{\sigma}_k(\Omega, \rho) = \mathfrak{S}_k(\Omega)$ is the supremum and $\int_{\partial\Omega}\tilde{\rho}_\epsilon \mathrm{d}\mathcal{H}^{1} \to \int_{\partial\Omega}\rho \mathrm{d}\mathcal{H}^{1}$,
    we conclude that (see also \cite[Eq.~(2.4) and above]{Vinokurov:2025:sym-eigen-val-lms})
    \begin{equation}
      \overbar{\sigma}_k(\Omega, \tilde{\rho}_\epsilon) \to \mathfrak{S}_k(\Omega).
    \end{equation}
    Arguing as in the proof of \cite[Proposition~4.3]{Vinokurov:2026:p-harm-and-conf-class-opt-gafa}, we use Ekeland's variational principle
    to construct an improved maximizing sequence $\brc{\rho_\epsilon}\subset L^\infty(\partial\Omega)$ together with a sequence of
    maps $u_\epsilon \in H^1(\Omega,\R^n)$ (see also the proof of \cite[Proposition~3.1]{Vinokurov:2025:sym-eigen-val-lms}) so that
    \begin{itemize}

      \item $\norm{\tilde{\rho}_\epsilon - \rho_\epsilon}_{L^1}\to 0$ and $\overbar{\sigma}_k(\Omega, \rho_\epsilon) \to \mathfrak{S}_k(\Omega)$;

      \item $\abs{u_\epsilon} \leq 1$ a.e. on $\Omega$ and the norms $\norm{u_\epsilon}_{H^1(\Omega)}$ are uniformly bounded;

      \item $\Delta u_\epsilon = \sigma_\epsilon u_\epsilon \mu_\epsilon$ in $H^1(\Omega)^*$,
      where $\mu_\epsilon := \rho_\epsilon \mathrm{d}\mathcal{H}^{1}|_{\partial \Omega}$ and $\sigma_\epsilon := \sigma_k(\Omega,\rho_\epsilon)$;

      \item $\abs{u_\epsilon}^2 \to 1$ in $L^1(\partial\Omega)$.

    \end{itemize}

    Note that the dimension $n$ of the target space is the same for all $u_\epsilon$, which is possible due to uniform multiplicity
    bounds when $\dim \Omega = 2$ (see, for example, \cite{Karpukhin-Kokarev-Polterovich:2014:steklov-multiplicity-bounds}); note also
    that the bound $\abs{u_\epsilon} \leq 1$ is first established $\mathcal{H}^1$-a.e. on $\supp \rho_\epsilon \subset \partial \Omega$,
    and then it extends a.e. on the whole $\Omega$ by the following weak maximum principle: for any measure $\mu \in \mathcal{M}_+(\overbar{\Omega})\cap \mathfrak{Bil}[H^1(\Omega)]$,
    \begin{equation}
      \Delta u =  u \mu \implies \Delta \abs{u} \leq \abs{u} \mu \ \text{ weakly},
    \end{equation}
    that is, $\int_{\Omega} \brt{\mathrm{d}\abs{u},\mathrm{d}\phi} \leq \int_{\overbar{\Omega}}\phi\abs{u}\mathrm{d}\mu$ for every $\phi \in H^1(\Omega),\ \phi \ge 0$.

    Since $\rho_\epsilon\to\rho$ in $L^1(\partial\Omega)$, we have $\mu_\epsilon \to \rho\mathrm{d}\mathcal{H}^1|_{\partial\Omega}$
    in total variation, and hence also weak-$*$. Once we have such a sequence $(\mu_\epsilon,u_\epsilon)$, we apply
    \cite[Lemmas~3.6 and~3.8]{Vinokurov:2025:sym-eigen-val-lms} to
    extract a subsequence such that
    $u_\epsilon \oset{w}{\to} u_*$ in $H^1$, where $u_*$ is a
    smooth free boundary harmonic map into the ball. Moreover, the limiting
    measure is induced by $u_*$, that is,
    \(
      \sigma_k(\Omega,\rho)\rho
      =
      \brt{u_*,\partial_\nu u_*}.
    \)
  \end{proof}
\subsubsection{Proof of Theorem~\ref{thm:simply-connected-steklov}}
  Observe that $\overbar{\Omega}$ is a compact simply connected topological surface with boundary,
  so it is homeomorphic to the closed disk. In particular, $\partial\Omega$ is a simple closed curve, which is rectifiable by assumption.
  Then the Riesz--Privalov theorem (see \cite[Theorem 6.8]{Pommerenke:1992:conf-maps}) tells us
  that there exists a conformal homeomorphism
  \begin{equation}
    f \colon \overbar{\mathbb{D}}^2 \to \overbar{\Omega}
  \end{equation}
  such that $f|_{\partial \mathbb{D}^2} \colon \partial\mathbb{D}^2 \to \partial \Omega$ is absolutely continuous, that is,
  $f^{-1}_*(\mathcal{H}^1 |_{\partial \Omega})= \abs{f'(e^{i\theta})}\mathrm{d}\theta$. Note also that there is a bijection
  $f^* \colon H^1(\Omega)\cap C^0(\overbar{\Omega}) \to H^1(\mathbb{D}^2)\cap C^0(\overbar{\mathbb{D}}^2)$ so that
  \begin{equation}\label{eq:stek-mapping-equality}
    \sigma_k(\Omega) = \inf_{V_{k+1}} \sup_{\phi \in V_{k+1}\setminus\brc{0}} \frac{\int_{\Omega} \abs{\mathrm{d}\phi}^2}{\int_{\partial \Omega} \phi^2 d\mathcal{H}^1}
    = \inf_{V_{k+1}} \sup_{\phi \in V_{k+1}\setminus\brc{0}} \frac{\int_{\mathbb{D}^2} \abs{\mathrm{d}(f^*\phi)}^2}{\int_{\partial\mathbb{D}^2} (f^*\phi)^2 \rho \mathrm{d}\theta}
    = \sigma_k(\mathbb{D}^2,\rho),
  \end{equation}
  where $\rho := \abs{f'(e^{i\theta})} \in L^1(\partial\mathbb{D}^2)$ and $V_{k+1} \subset H^1(\Omega)\cap C^0(\overbar{\Omega})$.

  It follows, for example, from \cite{Girouard-Polterovich:2010:payne-schiffer} and \cite{Vinokurov:2025:sym-eigen-val-lms} that
  \begin{equation}
    \mathfrak{S}_k(\mathbb{D}^2) = 2\pi k.
  \end{equation}
  Since $|\partial\Omega| = \int_{\partial\mathbb{D}^2} \rho \mathrm{d}\theta$, \eqref{eq:stek-mapping-equality} implies $\overbar{\sigma}_k(\Omega) = \overbar{\sigma}_k(\mathbb{D}^2,\rho)$.
  So, if $\overbar{\sigma}_k(\Omega) = \overbar{\sigma}_k(\mathbb{D}^2,\rho) = 2\pi k = \mathfrak{S}_k(\mathbb{D}^2)$,
  Lemma~\ref{lem:qualitative-steklov-stability} implies that $\rho \in C^\infty(\partial \mathbb{D}^2)$, and the smooth case
  was already established (see \cite[Remark~1.12]{Vinokurov:2025:sym-eigen-val-lms}).

  In particular, up to a conformal automorphism of the disk, $\overbar{\sigma}_1(\Omega) = \overbar{\sigma}_1(\mathbb{D}^2,\rho) = 2\pi$ is attained only for constant densities $\rho = \abs{f'|_{\partial\mathbb{D}^2}}$.
  By rescaling $\Omega$, assume that $\abs{f'|_{\partial\mathbb{D}^2}} \equiv 1$. We prove that $f$ is, in fact, an affine map, $f(z) = a z + b$, in that case, and hence, $\Omega$ is a disk.
  The idea is to show that $f'$ is continuous up to the boundary and then apply the maximum principle.

  By the Riesz–Privalov theorem, $f'$ belongs to the analytic Hardy space on $\mathbb{D}^2$. Since $\abs{f'}=1$ a.e. on $\partial\mathbb{D}^2$, the Hardy-space maximum principle gives $\abs{f'}\leq 1$ in $\mathbb{D}^2$.
  Let $F(\theta) := f(e^{i\theta})$.
  Around each point $F(\theta_0) \in \partial\Omega$, the boundary $\partial\Omega = f(\partial\mathbb{D}^2)$ is locally given as a graph of a continuous function, say, $h(x)$.
  Thus,
  \begin{equation}
    F(I) = F(\theta_0) + c_1\set{(x, h(x))}{x \in J}
  \end{equation}
  for small intervals $I\ni\theta_0$ and $J$, and some
  $c_1\in\C$ with $\abs{c_1}=1$.
  Hence, $\theta\mapsto\operatorname*{Re}
  c_1^{-1}f(e^{i\theta})$ is monotone. After shrinking $I$ if
  necessary, there exist $c_2\in\C$, $\abs{c_2}=1$, and
  $\epsilon>0$ sufficiently small such that
  \[
      \operatorname*{Re}c_2 f'(e^{i\theta})\ge-\epsilon
      \quad\text{for a.e. }\theta\in I.
  \]
  So, $f'(e^{i\theta})|_I$ does not cover the whole
  $\partial\mathbb{D}^2$, and since $\theta_0$ was arbitrary,
  $f'$ extends holomorphically across the entire circle by
  \cite[Theorem~4]{Seidel:1934:semicerc-extention}.
  Since the extension is continuous on $\overbar{\mathbb{D}}^2$, we obtain
  $\abs{f'}\equiv 1$ on $\partial\mathbb{D}^2$.
  Moreover, $f'$ does not vanish in $\overbar{\mathbb{D}}^2$. Thus, both $f'$ and $1/f'$ are
  holomorphic near $\overbar{\mathbb{D}}^2$.
  Applying the maximum principle to $f'$ and $1/f'$, we obtain
  $\abs{f'}\equiv 1$ on $\mathbb{D}^2$. Therefore,
  $f'\equiv a$ for some $a\in\partial\mathbb{D}^2$, and hence
  $f(z)=az+b$.

\printbibliography
\end{document}